\documentclass{article}
\pdfoutput=1

\usepackage{standard}
\usepackage{adjustbox}

\usepackage{tikz}
\usetikzlibrary{calc}

\usepackage[linesnumbered,ruled,vlined]{algorithm2e}
\SetKwInput{Input}{Input}
\SetKwInput{Output}{Output}

\usepackage{xcolor}

\usepackage{amsthm}

\newtheorem{problem}[theorem]{Problem}
\newtheorem{conjecture}[theorem]{Conjecture}

\title{Recovering a group from few orbits}
\author{Dustin G.\ Mixon\thanks{Department of Mathematics, The Ohio State University, Columbus, Ohio, USA} \thanks{Translational Data Analytics Institute, The Ohio State University, Columbus, Ohio, USA} \and Brantley Vose\footnotemark[1]}
\date{}

\begin{document}
\maketitle

\begin{abstract}
For an unknown finite group $G$ of automorphisms of a finite-dimensional Hilbert space, we find sharp bounds on the number of generic $G$-orbits needed to recover $G$ up to group isomorphism, as well as the number needed to recover $G$ as a concrete set of automorphisms.
\end{abstract}

\section{Introduction}

Symmetry is a powerful tool in mathematics, but in some applications, the symmetries at play are not self-evident.
Suppose a set has some unknown symmetry group acting upon it and we can only observe a few of the action's orbits, delivered as unstructured subsets.

\begin{center}
\textit{To what extent do the orbits of an action betray the underlying group?}
\end{center}

Motivated by the task of learning symmetries in data, we are particularly interested in determining an unknown finite group of automorphisms (i.e., linear isometries) of a finite-dimensional Hilbert space from a sample of its orbits.
In this setting, there are a few ways in which orbits can be uninformative.
For example, since $\{0\}$ is an orbit under all of the groups we consider, observing this orbit is not helpful.
Similarly, given two orbits that differ by a scalar multiple, observing both orbits is informationally equivalent to observing just one of them.
To avoid such degeneracies, we will assume that the orbits we receive are \emph{generic} in some sense. 
In particular, we will use the topologist's notion of genericity, which we make explicit in Section~\ref{subsec:conventions}.
Specifically, we would like to solve the following inverse problems.

\begin{problem}
\label{prob.main}
Let $V$ denote a finite-dimensional real or complex Hilbert space, and take any finite subgroup $G$ of the group $\operatorname{Aut}(V)$ of automorphisms of $V$. 
\begin{enumerate}[label=(\alph*)]
\item\label{prob.main.abstract}
\textbf{Abstract Group Recovery.}\ How many generic $G$-orbits determine $G$ up to group isomorphism?
\item\label{prob.main.concrete}
\textbf{Concrete Group Recovery.}\ How many generic $G$-orbits determine $G$ as a subset of $\operatorname{Aut}(V)$?
\end{enumerate}
\end{problem}

The reader is invited to try their hand at a few simple instances of Problem~\ref{prob.main} when $V=\mathbb R^2$, presented in Figure~\ref{fig:orbit-game}.
In this paper, we partially solve Problem~\ref{prob.main} in general, and we completely solve part~\ref{prob.main.abstract} in the complex case.
The remainder of this section discusses some relevant background before outlining the paper.

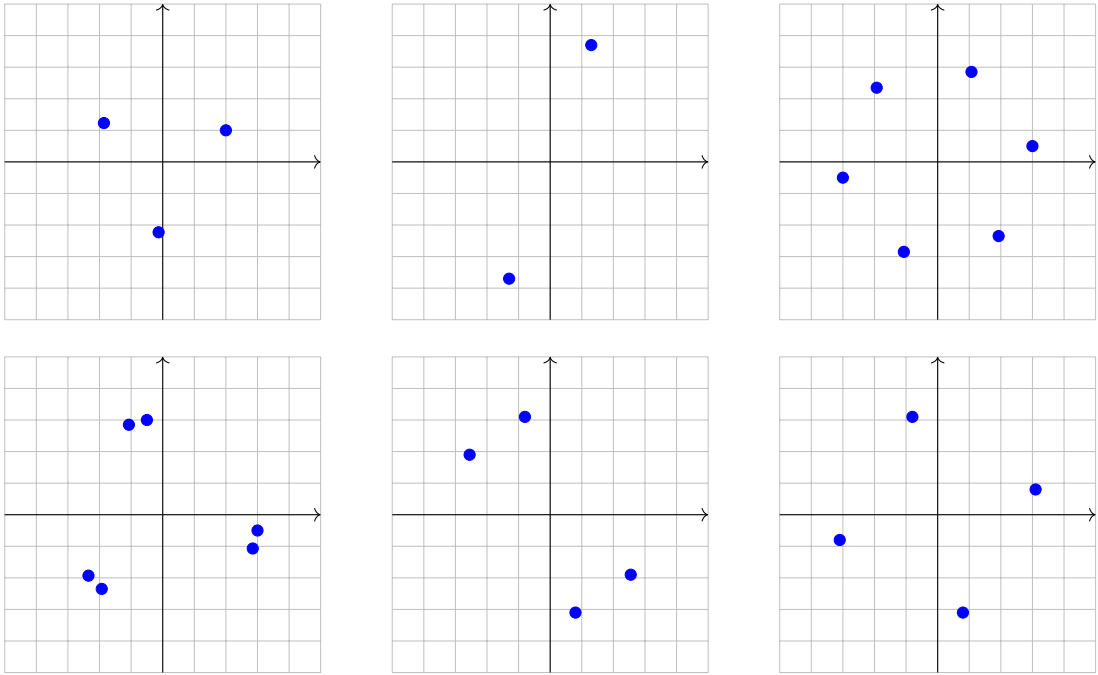
\begin{figure}
    \centering
    
    \begin{tikzpicture}[scale=0.42]
    \draw[very thin, lightgray] (-5, -5) grid (5, 5);

    \draw[->] (-5, 0) -- (5, 0) node[right] {};
    \draw[->] (0, -5) -- (0, 5) node[above] {};

    \coordinate (P1) at (2, 1);
    \coordinate (P2) at (-1.86, 1.23);
    \coordinate (P3) at (-0.13, -2.23);

    \filldraw[blue] (P1) circle (5pt);
    \filldraw[blue] (P2) circle (5pt);
    \filldraw[blue] (P3) circle (5pt);

    \end{tikzpicture}%
    \qquad%
    \begin{tikzpicture}[scale=0.42]
    \draw[very thin, lightgray] (-5, -5) grid (5, 5);

    \draw[->] (-5, 0) -- (5, 0) node[right] {};
    \draw[->] (0, -5) -- (0, 5) node[above] {};

    \coordinate (P1) at (1.3, 3.7);
    \coordinate (P2) at (-1.3, -3.7);

    \filldraw[blue] (P1) circle (5pt);
    \filldraw[blue] (P2) circle (5pt);

    \end{tikzpicture}%
    \qquad%
    \begin{tikzpicture}[scale=0.42]
    \draw[very thin, lightgray] (-5, -5) grid (5, 5);

    \draw[->] (-5, 0) -- (5, 0) node[right] {};
    \draw[->] (0, -5) -- (0, 5) node[above] {};

    \coordinate (P1) at (3, 0.5);
    \coordinate (P2) at (1.07, 2.85);
    \coordinate (P3) at (-1.93, 2.35);
    \coordinate (P4) at (-3, -0.5);
    \coordinate (P5) at (-1.07, -2.85);
    \coordinate (P6) at (1.93, -2.35);

    \filldraw[blue] (P1) circle (5pt);
    \filldraw[blue] (P2) circle (5pt);
    \filldraw[blue] (P3) circle (5pt);
    \filldraw[blue] (P4) circle (5pt);
    \filldraw[blue] (P5) circle (5pt);
    \filldraw[blue] (P6) circle (5pt);

    \end{tikzpicture}%

    \medskip
    
    \begin{tikzpicture}[scale=0.42]
    \draw[very thin, lightgray] (-5, -5) grid (5, 5);

    \draw[->] (-5, 0) -- (5, 0) node[right] {};
    \draw[->] (0, -5) -- (0, 5) node[above] {};

    \coordinate (P1) at (3, -0.5);
    \coordinate (P2) at (-0.5, 3);
    \coordinate (P3) at (-1.07, 2.85);
    \coordinate (P4) at (2.85, -1.07);
    \coordinate (P5) at (-1.93, -2.35);
    \coordinate (P6) at (-2.35, -1.93);

    \filldraw[blue] (P1) circle (5pt);
    \filldraw[blue] (P2) circle (5pt);
    \filldraw[blue] (P3) circle (5pt);
    \filldraw[blue] (P4) circle (5pt);
    \filldraw[blue] (P5) circle (5pt);
    \filldraw[blue] (P6) circle (5pt);

    \end{tikzpicture}%
    \qquad%
    \begin{tikzpicture}[scale=0.42]
    \draw[very thin, lightgray] (-5, -5) grid (5, 5);

    \draw[->] (-5, 0) -- (5, 0) node[right] {};
    \draw[->] (0, -5) -- (0, 5) node[above] {};

    \coordinate (P1) at (0.8, -3.1);
    \coordinate (P2) at (2.55, -1.9);
    \coordinate (P3) at (-0.8, 3.1);
    \coordinate (P4) at (-2.55, 1.9);

    \filldraw[blue] (P1) circle (5pt);
    \filldraw[blue] (P2) circle (5pt);
    \filldraw[blue] (P3) circle (5pt);
    \filldraw[blue] (P4) circle (5pt);

    \end{tikzpicture}%
    \qquad%
    \begin{tikzpicture}[scale=0.42]
    \draw[very thin, lightgray] (-5, -5) grid (5, 5);

    \draw[->] (-5, 0) -- (5, 0) node[right] {};
    \draw[->] (0, -5) -- (0, 5) node[above] {};

    \coordinate (P1) at (-0.8, 3.1);
    \coordinate (P2) at (3.1, 0.8);
    \coordinate (P3) at (0.8, -3.1);
    \coordinate (P4) at (-3.1, -0.8);

    \filldraw[blue] (P1) circle (5pt);
    \filldraw[blue] (P2) circle (5pt);
    \filldraw[blue] (P3) circle (5pt);
    \filldraw[blue] (P4) circle (5pt);

    \end{tikzpicture}%
    \qquad%
    
    \caption{Six orbits in $\bbR^2$ arising from the actions of six different subgroups of the orthogonal group $\op{O}(2)$. The reader is invited to guess the isomorphism class of the group that generated each orbit. In each case, you may assume that the point that generated the orbit was drawn at random according to a continuous probability distribution over $\mathbb R^2$. The solutions can be found in the footnote on the next page.\protect\footnotemark}
    \label{fig:orbit-game}
\end{figure}

\subsection{Related work}

This paper represents the latest entry in a rapidly growing literature on symmetry in data science.
For context, we describe some of the most recent papers in this space, which can be roughly partitioned according to whether the motivating application stems from signal processing or machine learning.

\textbf{Signal processing.}
In many applications, one must reconstruct an object from various observations despite an ambiguity that stems from a known group action.
In phase retrieval~\cite{BendoryDES:23}, one seeks to recover a function with some known structure from the pointwise modulus of its Fourier transform, thereby exhibiting an ambiguity from $\operatorname{U}(1)$.
In cryogenic electron microscopy~\cite{BendoryDES:24,BendoryE:24,BendoryEM:24,EdidinS:24,FanLSWX:24,HoskinsKMSW:24,ZhangMKVMGS:24}, one seeks to reconstruct a function in $L^2(\mathbb{R}^3)$ from several noisy $2$-dimensional projections of randomly rotated versions of the function.
To facilitate the study of cryo-EM, it has become popular to consider sub-problems like \textit{synchronization} and \textit{multi-reference alignment}. 
In synchronization~\cite{HadiBS:24}, one seeks to reconstruct group elements that secretly label the vertices of a directed graph from noisy quotients that correspond to directed edges between the vertices.
In multi-reference alignment~\cite{BandeiraBSKNWPW:23,CahillCC:20,EdidinK:24,YinLH:24}, one seeks to reconstruct a function from several noisy and randomly translated versions of the function.

\textbf{Machine learning.}
Various tasks in machine learning can be made more efficient after identifying a group that acts on the input data.
In such settings, depending on the task, it can be beneficial to use features that are either invariant or equivariant to the group action~\cite{BlumSmithV:23}.
Recently, there has been a flurry of research~\cite{CahillCC:24,CahillIM:24,Derksen:24,DymG:24} to upgrade classical invariant theory to allow for invariant features that are semialgebraic, Lipschitz, and sometimes even bilipschitz in the quotient metric.
Some of the most popular invariants for this task are based on \textit{max filtering}~\cite{CahillIMP:24,MixonP:23,MixonQ:22} or \textit{coorbits}~\cite{BalanT:23a,BalanT:23b,BalanTW:24,MixonQ:24}.
In contexts where a data point is the adjacency matrix of a weighted graph on $n$ vertices, one might account for the conjugation action of $S_n$ using a graph neural network or similar architecture~\cite{AmirGARD:24,BlumSmithHCV:24,BokerLHVM:24,HordanAD:24,HuangLV:24,SverdlovDDA:24,SverdlovSD:24}.
When learning physical laws from data, it can be helpful to account for equivariance to rotation~\cite{VillarHSFYBS:21} or to changes in the units~\cite{VillarYHBSD:23}.
Finally, there are some settings in which the group is not known \textit{a priori}, and so one must learn the group from the data~\cite{CahillMP:23,EnnesT:23}.
The present paper is motivated by this same problem, although we focus on finite groups, whereas the previous papers focused on connected Lie groups.

In terms of mathematical techniques, this paper is a spiritual descendant of two different lines of research, which focus on particular applications of symmetry and genericity, respectively:

\textbf{Symmetry in codes.}
In~\cite{FejesToth:64,FejesToth:86}, Fejes T\'{o}th observed that solutions to extremal problems in geometry frequently exhibit high degrees of symmetry.
This occurs, for example, when packing points in the sphere in such a way that maximizes the minimum pairwise distance~\cite{FejesToth:40}.
In fact, there are several \textit{universally optimal} configurations in $S^{d-1}$~\cite{CohnK:07}, meaning they simultaneously minimize a particular infinite family of energies, and each of these configurations happens to be highly symmetric.
The optimal codes of $d^2$ points in $\mathbb{CP}^{d-1}$ are conjecturally given by Heisenberg--Weyl orbits that enjoy an efficient description in terms of Stark units~\cite{Kopp:18}.
There are many more instances of symmetric arrangements being optimal as codes~\cite{CoxKMP:20,FickusGI:24,FickusIJM:24,IversonJM:24,IversonM:22,IversonM:24}.
The symmetry exhibited by a given arrangement of points was studied by Waldron and his collaborators in~\cite{BroomeW:13,ChienW:16,KingMW:21,ValeW:10,ValeW:04,Waldron:18}, and their use of graphs and representations greatly influenced the approach we use in this paper.

\textbf{Generic identifiability.}
Given a function $f\colon X\to Y$ and data $y\in\operatorname{im}f$, consider the inverse problem of finding all $x\in X$ for which $f(x)=y$.
In settings where $f$ is semialgebraically specified by a vector of parameters $z$, one might ask whether this inverse problem has a unique solution for a generic $z$ and arbitrary $y$ (or for a generic combination of $z$ and $y$).
This approach was first adopted by Balan, Casazza, and Edidin in~\cite{BalanCE:06} to determine injectivity conditions for phase retrieval.
The techniques developed in that paper have since been used in other settings related to phase retrieval~\cite{ConcaEHV:15,RongWX:21,WangX:19} and invariant machine learning~\cite{CahillCC:20,CahillIMP:24,DymG:24}.
Considering the prominence of the word ``generic'' in Problem~\ref{prob.main}, it should come as no surprise that these techniques also appear in this paper.

\subsection{Conventions and preliminaries}\label{subsec:conventions}
If a group $G$ acts on a set $X$, we use $Gx$ to denote the orbit of a point $x\in X$. 
Given a subset $S$ of a real or complex Hilbert space $V$, we let $\op{Aut}(S)$ denote the group of permutations $\pi\colon S\to S$ for which there exists a linear isometry $M\colon V\to V$ with $M|_S = \sigma$. 
That is, $\op{Aut}(S)$ is the group of permutations on $S$ that extend (possibly non-uniquely) to a linear isometry of $V$. 
In particular, $\op{Aut}(V)$ is the unitary group $\op{U}(V)$ if $V$ is a complex, and the orthogonal group $\op{O}(V)$ if $V$ is real.
Throughout this paper, the inner product of a complex Hilbert space is conjugate-linear in the \emph{first} argument and linear in the \emph{second}. 

We will make use of graphs as convenient data structures. Specifically, by \emph{edge-labeled directed graph}, we mean a directed graph with no multiple edges and whose edges are labeled with elements of some underlying edge label set.
Given edge-labeled directed graphs $\Gamma_1$ and $\Gamma_2$ with vertex sets $V_1$ and $V_2$ and edge label sets $L_1$ and $L_2$ respectively, an \emph{isomorphism of edge-labeled directed graphs $\Gamma_1\to \Gamma_2$} is a pair of bijections $f_V\colon V_1\to V_2$ and $f_L\colon L_1\to L_2$ such that
\begin{itemize}
    \item For each $v,w\in V_1$, $\Gamma_1$ has a directed edge $v\to w$ if and only if $\Gamma_2$ has a directed edge $f_V(v)\to f_V(w)$.
    \item For each edge $v\to w$ in $\Gamma_1$ with label $\ell$, the edge $f_V(v)\to f_V(w)$ in $\Gamma_2$ is labeled with $f_L(\ell)$.
\end{itemize}
In particular, an isomorphism of edge-labeled directed graphs need not preserve the edge labels, and the graphs involved need not share an edge label set.

Given a topological space $X$, a subset $U \subseteq X$ is said to be \emph{generic} if $U$ is open and dense in $X$. 
A condition on the elements of $X$ is called a \emph{generic condition} if there is a generic subset of $X$ on which the condition holds.
We use two primary tools to construct generic sets.
First, the complement of the zero set of a nonzero polynomial function $p\colon \mathbb{R}^n\to\mathbb{R}$ is generic; indeed, given a polynomial $p$ that is identically zero over a neighborhood of a point $x_0\in\mathbb{R}^n$, its Taylor series expansion about $x_0$ reveals that $p$ is the zero polynomial.
Second, any finite intersection of generic sets is generic.
In particular, every generic set we construct in this paper is a nonempty open set in the \textit{Zariski topology}. 
For a complex vector space, we use the Zariski topology that arises from viewing it as a real vector space.

\footnotetext{Solutions for Figure~\ref{fig:orbit-game}: The underlying groups are isomorphic to (top row, left to right) $C_3$, $C_2$, $C_6$, (bottom row, left to right) $D_3$, $D_2$, $C_4$. Note that some orbits could conceivably have been generated by multiple distinct groups. For instance, the hexagonal orbit could have been generated by a group isomorphic to $D_3$ or even $D_6$. However, since the initial point was drawn at random according to a continuous probability distribution, $D_3$ or $D_6$ would produce a hexagonal orbit with probability zero, so we may safely rule them out. We address these issues of coincidences with more care in Section~\ref{sec.abstract recovery}.}

\subsection{Roadmap}

In the next section, we tackle Problem~\ref{prob.main}\ref{prob.main.abstract}.
Perhaps surprisingly, the number of necessary orbits does not depend on the dimension of $V$ or any property of $G$, at least in the complex case: a \textit{single} generic orbit already betrays the isomorphism class of $G$.
One can recover a group isomorphic to $G$, for instance, as the automorphism group of the orbit.
We also find that in the real case, a similar method recovers the isomorphism class of $G$ provided that one observes \emph{two} generic orbits.
We suspect only one generic orbit is necessary in the real case, and that our requirement for a second orbit is an artifact of our recovery method.
We show that a single generic orbit suffices in the real case if $G$ has prime order or if $V$ has dimension $2$, but we leave the general case as an open problem.

In Section~\ref{sec:concrete recovery} we turn our attention to Problem~\ref{prob.main}\ref{prob.main.concrete}. 
We recover the concrete group $G$ in two steps. 
We first show that, given enough generic orbits, one can recover the permutations that $G$ induces on the union of orbits.
Next, if the orbits span a large enough subspace of $V$, then we can recover $G$ as a concrete group by linearly extending the permutation action to the span of the orbits, and then arguing that $G$ necessarily acts trivially on the orthogonal complement.
We derive explicit bounds on the number of necessary generic orbits using tools from representation theory.
We conclude in Section~\ref{sec.discussion} with a brief discussion.

Table~\ref{table.bounds} summarizes our main results, namely, numbers of generic orbits that suffice to solve Problem~\ref{prob.main} in various settings.
\begin{table}[h]
\centering
\begin{tabular}{|lcc|}
\hline
$\phantom{\bigg|}$
underlying field & $\mathbb C$ & $\mathbb R$
\\ \hline\hline
$\phantom{\bigg|}$
abstract group recovery & 1 & 2 
\\ \hline
$\phantom{\bigg|}$
concrete group recovery \quad & \quad $\displaystyle\max_\pi \tfrac{n_\pi(V)-(r-1)\cdot[\pi=\mathbf{1}]}{\operatorname{dim} \pi}$ \quad & \quad $\displaystyle\max\left\{~\max_\pi \tfrac{n_\pi(V)-(r-1)\cdot[\pi=\mathbf{1}]}{n_\pi(R)},~ 2~\right\}$ \quad \\\hline
\end{tabular}
\caption{The main results of this paper. Each entry reports how many generic orbits suffice to solve the abstract or concrete group recovery problem; see Problem~\ref{prob.main}. The rows represent parts~\ref{prob.main.abstract} and \ref{prob.main.concrete} of Problem~\ref{prob.main}, respectively. The columns represent the field underlying the Hilbert space $V$ in question. The notation in the second row is explained in detail in Section~\ref{sec:concrete recovery}, but we explain it briefly here. The maximum is over all (finitely many) irreducible representations $\pi$ of the hidden abstract group, $n_\pi(V)$ denotes the multiplicity of $\pi$ in $V$, $R$ denotes the regular representation, $\mathbf{1}$ denotes the trivial representation, $r$ is the minimum dimension of a nontrivial representation of the hidden abstract group, and the \textit{Iverson bracket} $[P]$ is $1$ when the statement $P$ holds and $0$ otherwise.} \label{table.bounds}
\end{table}

\section{Recovering the abstract group}\label{sec.abstract recovery}

We start by drawing some intuition from low-dimensional examples.

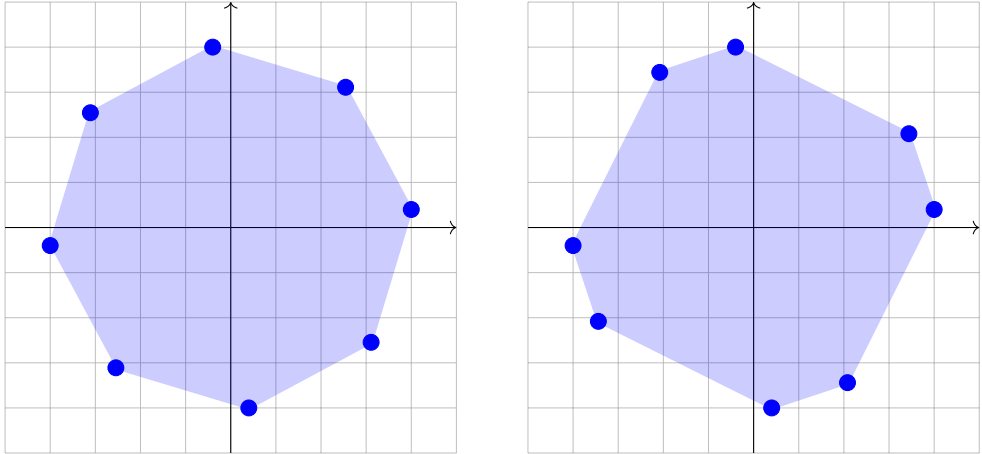
\begin{figure}
    \centering

    \begin{tikzpicture}[scale=0.6]
    \draw[very thin, lightgray] (-5, -5) grid (5, 5);

    \draw[->] (-5, 0) -- (5, 0) node[right] {};
    \draw[->] (0, -5) -- (0, 5) node[above] {};

    \coordinate (P1) at (0.4, -4);
    \coordinate (P2) at (3.11, -2.545);
    \coordinate (P3) at (4, 0.4);
    \coordinate (P4) at (2.545, 3.11);
    \coordinate (P5) at (-0.4, 4);
    \coordinate (P6) at (-3.11, 2.545);
    \coordinate (P7) at (-4, -0.4);
    \coordinate (P8) at (-2.545, -3.11);

    \draw[very thick, blue, fill, opacity = 0.2] (P8) -- (P1) -- (P2) -- (P3) -- (P4) -- (P5) -- (P6) -- (P7);

    \filldraw[blue] (P1) circle (5pt);
    \filldraw[blue] (P2) circle (5pt);
    \filldraw[blue] (P3) circle (5pt);
    \filldraw[blue] (P4) circle (5pt);
    \filldraw[blue] (P5) circle (5pt);
    \filldraw[blue] (P6) circle (5pt);
    \filldraw[blue] (P7) circle (5pt);
    \filldraw[blue] (P8) circle (5pt);

    \end{tikzpicture}%
    \qquad%
    \begin{tikzpicture}[scale=0.6]
    \draw[very thin, lightgray] (-5, -5) grid (5, 5);

    \draw[->] (-5, 0) -- (5, 0) node[right] {};
    \draw[->] (0, -5) -- (0, 5) node[above] {};

    \coordinate (P1) at (0.4, -4);
    \coordinate (P2) at (3.44, 2.08);
    \coordinate (P3) at (4, 0.4);
    \coordinate (P4) at (2.08, -3.44);
    \coordinate (P5) at (-0.4, 4);
    \coordinate (P6) at (-3.44, -2.08);
    \coordinate (P7) at (-4, -0.4);
    \coordinate (P8) at (-2.08, 3.44);

    \draw[very thick, blue, fill, opacity=0.2] (P8) -- (P5) -- (P2) -- (P3) -- (P4) -- (P1) -- (P6) -- (P7) -- (P8);

    \filldraw[blue] (P1) circle (5pt);
    \filldraw[blue] (P2) circle (5pt);
    \filldraw[blue] (P3) circle (5pt);
    \filldraw[blue] (P4) circle (5pt);
    \filldraw[blue] (P5) circle (5pt);
    \filldraw[blue] (P6) circle (5pt);
    \filldraw[blue] (P7) circle (5pt);
    \filldraw[blue] (P8) circle (5pt);

    \end{tikzpicture}%
    
    \caption{Orbits generated by $C_8$ (left) and $D_4$ (right). As discussed in Example~\ref{ex:davinci-actions}, the points fall on the vertices of a centered regular octagon and a centered truncated square, respectively, each shown in blue.}
    \label{fig:davinci-examples}
\end{figure}

\begin{example}\label{ex:davinci-actions}
Suppose $V=\mathbb{R}^2$.
By a theorem of Leonardo da Vinci~\cite{Weyl:52}, every subgroup of $\operatorname{O}(2)$ is either cyclic $C_n\leq\operatorname{SO}(2)$ or dihedral $D_n=\langle C_n,r\rangle$, where $r$ denotes reflection about some $1$-dimensional subspace of $\mathbb{R}^2$.
For $n\geq 3$, a generic orbit of $C_n$ forms the vertices of a regular $n$-gon. Meanwhile, for $n\geq 3$, a generic orbit of $D_n$ forms the vertices of a truncated regular $n$-gon, and a generic orbit of $D_2$ forms the vertices of an oblong rectangle. (See Figure~\ref{fig:davinci-examples} for an illustration.) Finally, both $C_2$ and $D_1$ have pairs of points as their generic orbits.
Considering $C_2\cong D_1$, it follows that the shape of a single orbit betrays the underlying group up to isomorphism, at least when $V=\mathbb R^2$.
\end{example}

Following the spirit of the above example, we attempt to use the shape of a single generic orbit to determine the underlying group.
First, we note that in the above example, the size of a generic orbit betrays the order of the group.
This behavior generalizes:

\begin{lemma}
\label{lem.distinct points}
Let $V$ denote a finite-dimensional real or complex vector space, and take any finite $G \leq \op{GL}(V)$.
Then for a generic $x\in V$, it holds that $|Gx|=|G|$.
\end{lemma}

\begin{proof}
Equip $V$ with an inner product, and given a non-identity group element $g\in G$, consider the function $p\colon x\mapsto\|x-gx\|^2$.
Considering $g$ is linear, if we view $V$ as a real vector space, then $p$ is a polynomial function of $\operatorname{dim}_\mathbb{R}(V)$ real coordinates.
Furthermore, since $g$ is not the identity, there exists $x_0$ such that $gx_0\neq x_0$, i.e., $p(x_0)\neq0$.
As such, $p$ is not the zero polynomial, and so there exists a generic set $K_g\subseteq V$ such that $gx\neq x$ for every $x\in K_g$.
It follows that $|G x|=|G|$ for every $x$ in the generic set $\bigcap_{g\in G\setminus\{1\}} K_g$.
\end{proof}

As a consequence of Lemma~\ref{lem.distinct points}, a group of prime order is betrayed by a single generic orbit, even if it does not act isometrically.
Next, to access the shape of a generic orbit, we note that shape is preserved by isometries.
As such, we are inclined to mod out by linear isometries by passing to the Gram matrix of the orbit.
However, since the orbit we receive is a set without additional structure, we do not have the information necessary to label the rows and columns of the Gram matrix by elements of the underlying group.
To avoid arbitrary indexing choices, we opt for an edge-labeled directed graph structure instead of a matrix:

\begin{definition}
Let $S$ be a finite subset of a Hilbert space $V$. 
The \textbf{Gram graph} of $S$ is the edge-labeled directed graph with vertex set $S$ such that for any vertices $s,t\in S$, there is a directed edge from $s$ to $t$ with the label $\langle s,t\rangle$.
\end{definition}

\begin{example}
\label{eq.c4 vs d2}
For $C_4$ and $D_2$ acting faithfully, linearly, and isometrically on $\mathbb{R}^2$, a generic orbit has a Gram graph with weighted adjacency matrix of the form
\[
C_4:
~~
\left[\begin{array}{cccc}
a&b&c&b\\
b&a&b&c\\
c&b&a&b\\
b&c&b&a
\end{array}\right],
\qquad\qquad
D_2:
~~
\left[\begin{array}{cccc}
a&b&c&d\\
b&a&d&c\\
c&d&a&b\\
d&c&b&a
\end{array}\right],
\]
with all distinct $a,b,c,d\in\mathbb{R}$.
We highlight two important facts about these examples:
\begin{itemize}
\item[(i)]
For each of these two groups, there is a single isomorphism class of edge-labeled directed graphs that contains the Gram graph of a generic orbit.
\item[(ii)]
The groups $C_4$ and $D_2$ determine distinct isomorphism classes of edge-labeled directed graphs.
\end{itemize}
These two facts together imply that the Gram graph of a generic orbit distinguishes $C_4$ from $D_2$.
\end{example}

The following theorem generalizes (i) in the above example.

\begin{theorem}\label{Gram-graph-characterization}
\label{thm.generic gram graph}
Let $V$ denote a finite-dimensional Hilbert space, and take any finite $G\leq \op{Aut}(V)$.
For a generic $x\in V$, the Gram graph of the orbit $Gx$ is isomorphic to the complete directed graph on vertex set $G$ with each edge $h\to k$ labeled by the corresponding linear map $\lambda(h^{-1}k)\colon V\to V$ defined by
\[
\lambda(g)
:=
\left\{\begin{array}{cl}
~g\phantom{^{-1}} & \text{if $V$ is complex,}\\
~g+g^{-1} & \text{if $V$ is real.}
\end{array}
\right.
\]
\end{theorem}

\begin{proof}
Consider the following intermediate claim:
\begin{quote}
For generic $x\in V$, the level sets of $g\mapsto\langle x,gx\rangle$ are identical to the level sets of $g\mapsto\lambda(g)$.
\end{quote}
First, we show that this claim implies the result.
By Lemma~\ref{lem.distinct points}, the map $g\mapsto gx$ is injective on $G$ for generic $x$.
Thus, we may relabel the vertices of the Gram graph of $Gx$ by the corresponding members of $G$.
Next, the edge $h\to k$ in the Gram graph is labeled by the scalar $\langle hx,kx\rangle=\langle x,h^{-1}kx\rangle$.
By the intermediate claim, we may relabel the edges by the corresponding linear maps $\lambda(h^{-1}k)$, and the result follows.
It remains to establish the intermediate claim, which we prove in cases.

\medskip

\noindent
\textbf{Case I:} $V$ is complex.
First, if $\lambda(h)=\lambda(k)$, then $h=k$, and so $\langle x,hx\rangle=\langle x,kx\rangle$ for all $x\in V$.
For the other direction, suppose $\lambda(h)\neq\lambda(k)$, i.e., $h\neq k$, and consider the unique decomposition $h-k=A+iB$ in terms of self-adjoint maps $A$ and $B$.
Since $h\neq k$, we have that $A$ or $B$ is nonzero, and so $\langle x,(h-k)x\rangle
=\langle x,Ax\rangle+i\langle x,Bx\rangle$ is nonzero for all $x$ in a generic set $K_{h,k}\subseteq V$.
Thus, for all $x\in\bigcap_{g,g'\in G,g\neq g'}K_{g,g'}$, it holds that $\langle x,hx\rangle\neq\langle x,kx\rangle$.

\medskip

\noindent
\textbf{Case II:} $V$ is real.
First, $\langle x,gx\rangle=\langle g^{-1}x,x\rangle=\langle x,g^{-1}x\rangle$, and so $\langle x,gx\rangle=\frac{1}{2}\langle x,\lambda(g)x\rangle$.
Thus, it suffices to show that the level sets of $g\mapsto\langle x,\lambda(g)x\rangle$ are identical to those of $g\mapsto\lambda(g)$.
Of course, $\lambda(h)=\lambda(k)$ implies $\langle x,\lambda(h)x\rangle=\langle x,\lambda(k)x\rangle$ for all $x\in V$.
For the other direction, suppose $\lambda(h)\neq\lambda(k)$.
Then since $\lambda(h)-\lambda(k)$ is self-adjoint, $\langle x,(\lambda(h)-\lambda(k))x\rangle$ is nonzero for all $x$ in a generic set $K_{h,k}\subseteq V$.
Thus, for all $x\in\bigcap_{g,g'\in G,\lambda(g)\neq\lambda(g')}K_{g,g'}$, it holds that $\langle x,hx\rangle\neq\langle x,kx\rangle$.
\end{proof}

Theorem~\ref{Gram-graph-characterization} allows us to recover the abstract group from the Gram graph of a single generic orbit in the complex case. 
It also empowers us to recover the abstract group as the automorphism group of a single generic orbit, a capability we will use extensively in Section~\ref{sec:concrete recovery}.
Both of these consequences are captured in the following result.

\begin{theorem}[One-orbit theorem]
\label{thm:one-orbit}
\label{cor.complex gram graph to G}
Let $V$ denote a finite-dimensional complex Hilbert space, and take any finite $G\leq \op{U}(V)$.
Then for a generic $v\in V$, each of the following holds:
\begin{enumerate}[label=(\alph*)]
    \item\label{item:one-orbit-gram-graph} The isomorphism class of the Gram graph of $Gv$ determines $G$ up to isomorphism.
    \item\label{item:one-orbit-automorphism} The canonical map $G\to \op{Aut}(Gv)$ is an isomorphism.
\end{enumerate}
\end{theorem}

\begin{proof}
For~\ref{item:one-orbit-gram-graph}, Theorem~\ref{thm.generic gram graph} implies that the Gram graph of a generic orbit is isomorphic to the complete directed graph on vertex set $G$ with each edge $h\to k$ labeled by the group element $h^{-1}k$.
This is known as the \textit{complete Cayley graph} of $G$.
The group of label-preserving graph automorphisms of the complete Cayley graph of $G$ is itself isomorphic to $G$, since each automorphism arises from left-multiplication by an element of $G$. The Gram graph of a generic orbit, being isomorphic to the complete Cayley graph of $G$, must exhibit the same automorphism group. Hence the label-preserving automorphism group of the Gram graph is isomorphic to $G$.

For~\ref{item:one-orbit-automorphism}, let $v\in V$ be generic.
By Lemma~\ref{lem.distinct points}, we may assume that $|Gv| = |G|$ and, since the action of $G$ on $Gv$ is transitive, that the canonical map $G\to \op{Aut}(Gv)$ is injective.
Each $\sigma \in \op{Aut}(Gv)$ preserves inner products between the elements of $Gv$ and hence induces a distinct label-preserving automorphism of the Gram graph of $Gv$. Since we have already shown in the proof of part~\ref{item:one-orbit-gram-graph} that there are exactly $|G|$ such automorphisms, we are done.
\end{proof}

Theorem~\ref{Gram-graph-characterization} seems to imply that the Gram graph of a single generic orbit in the real setting contains less information than in the complex setting. For instance, every skew-adjoint group element $g$ satisfies $\lambda(g) = g+g^{-1} = 0$. Even so, we believe that a real analog of Theorem~\ref{cor.complex gram graph to G}\ref{item:one-orbit-gram-graph} also holds, though a proof eludes us.

\begin{conjecture}
\label{conj.real gram graph to G}
Let $V$ denote a finite-dimensional real Hilbert space, and take any finite $G\leq \op{O}(V)$.
Then for a generic $x\in V$, the isomorphism class of the Gram graph of $Gx$ determines $G$ up to isomorphism.
\end{conjecture}

We illustrate the plausibility and difficulty of Conjecture~\ref{conj.real gram graph to G} with a couple of examples.

\begin{example}
\label{ex.real gram graph to G in R^2}
Conjecture~\ref{conj.real gram graph to G} holds in two dimensions as a consequence of Theorem~\ref{thm.generic gram graph}.
Indeed, for $C_n$, the Gram graph of a generic orbit has $n$ vertices and $\lfloor\frac{n}{2}\rfloor+1$ edge labels, whereas for $D_n$, the Gram graph of a generic orbit has $2n$ vertices and $\lfloor\frac{3n}{2}\rfloor+1$ edge labels.
These features only agree for $C_2$ and $D_1$, which in turn are isomorphic as groups.
\end{example}

\begin{example}
It is possible for distinct but isomorphic subgroups of orthogonal groups to exhibit different Gram graphs.
For example, consider the following subgroups of $\operatorname{O}(4)$, each isomorphic to $C_4\oplus C_2$:
\[
G_1:=\left\langle r\oplus \operatorname{id}, \operatorname{id}\oplus-\operatorname{id}\right\rangle,
\qquad
G_2:=\left\langle r\oplus r, \operatorname{id}\oplus-\operatorname{id}\right\rangle,
\qquad
r:=\left[\begin{array}{rr} 0&-1\\1&0\end{array}\right],
\qquad
\operatorname{id}:=\left[\begin{array}{cc} 1&0\\0&1\end{array}\right].
\]
It turns out that the Gram graph of a generic orbit under $G_1$ has $6$ distinct edge labels, whereas that the Gram graph under $G_2$ has only $5$.
In particular, the real case differs from the complex case in the sense that the conjectured map from Gram graph isomorphism class to group isomorphism class is necessarily many-to-one.
\end{example}

While the real analog of Theorem~\ref{thm:one-orbit}\ref{item:one-orbit-gram-graph} is left as a conjecture, a naively stated real analog of Theorem~\ref{thm:one-orbit}\ref{item:one-orbit-automorphism} is simply false.
That is, given a real finite-dimensional Hilbert space $V$ and a finite subgroup $G\leq \op{O}(V)$, the symmetry group of a generic orbit is often strictly larger than $G$.
\begin{example}
    In Example~\ref{ex:davinci-actions}, we observed that a generic orbit under the action of $C_n\leq \op{O}(2)$ forms the vertices of a regular $n$-gon.
    The automorphism group of such an orbit is of course isomorphic to the dihedral group $D_n$, of which the sought-after $C_n$ is a strict subgroup of index $2$.
\end{example}
A more extreme example can be found in four dimensions.
\begin{example}
    Consider the quaternions as a four-dimensional real vector space equipped with the inner product obtained by declaring $\{1,i,j,k\}$ to be orthonormal.
    The order-$8$ quaternion group $Q_8$ acts on the space of all quaternions by left-multiplication.
    A generic orbit of this action falls on the vertices of a regular four-dimensional cross-polytope (also called the \emph{$16$-cell} or the \emph{hexadecachoron}) centered at the origin.
    Hence for a generic quaternion $q$, $\op{Aut}(Q_8q)$ is isomorphic to the symmetry group of the $16$-cell, a group of order $384$, of which the far-smaller $Q_8$ is a subgroup of index $48$.
\end{example}

We can see that, in the real setting, the automorphism group of a generic orbit can be far larger than the acting group. Even so, we find that observing \emph{two} generic orbits is enough to break the extraneous symmetry and determine the isomorphism class of $G$, giving us a correct real analog of Theorem~\ref{thm:one-orbit}\ref{item:one-orbit-automorphism}.
\begin{theorem}[Two-orbit theorem]
\label{thm:two-orbit}
Let $V$ denote a finite-dimensional real Hilbert space, and take any finite $G\leq \op{O}(V)$.
Then for generic $(v,w)\in V$, the canonical map $G\to \op{Aut}(Gv \cup Gw)$ is an isomorphism.
In particular, two generic orbits determine $G$ up to isomorphism.
\end{theorem}

\begin{proof}
Assume  that $(v,w)\in V$ are such that $G\cong \op{Aut}(G(v + iw))$ (which is generic by Theorem~\ref{thm:one-orbit}) and $|Gv| = |Gw| = |G|$ (which is a generic by Lemma~\ref{lem.distinct points}).
The latter implies the canonical map $G \to \op{Aut}(Gv \cup Gw)$ is an injection, so we need only conclude that the latter group is of order at most $|G|$.
Since $G\cong \op{Aut}(G(v + iw))$, it is also sufficient to construct an injection
\[
    \alpha\colon \op{Aut}(Gv \cup Gw) \to \op{Aut}(G(v + iw)).
\]
We define $\alpha(\sigma)(x + iy) := \sigma(x) + i\sigma(y)$. Since $\sigma \in \op{Aut}(Gv \cup Gw)$, it extends to a linear isometry $M\colon V \to V$. Hence $\alpha(\sigma)$ extends to the linear isometry $x+iy\mapsto Mx + iMy$, demonstrating that $\alpha(\sigma)$ is indeed in $\op{Aut}(G(v + iw))$. We can also see that $\alpha$ is injective, since if $\alpha(\sigma) = \alpha(\tau)$, then in particular,
\begin{align*}
    \sigma(gv) + i\sigma(gw) = \alpha(\sigma)(v + iw) = \alpha(\tau)(v + iw) = \tau(gv) + i\tau(gw).
\end{align*}
By examining the real and imaginary parts of this equation, we find that $\sigma = \tau$.
\end{proof}

\section{Recovering the concrete group}\label{sec:concrete recovery}

We have seen that Problem~\ref{prob.main}\ref{prob.main.abstract} can be solved with one, possibly two generic orbits.
In this section we will find, as one might expect, that the stronger Problem~\ref{prob.main}\ref{prob.main.concrete} may require significantly more.
We will recover the concrete group $G$ in two steps.
In the first step, we determine how $G$ permutes the points in a union of several generic orbits.
In the second step, we linearly extend this action to the span of these orbits.
Given enough generic orbits, this will determine the full action of $G$.

\subsection{Recovering the action on generic orbits}

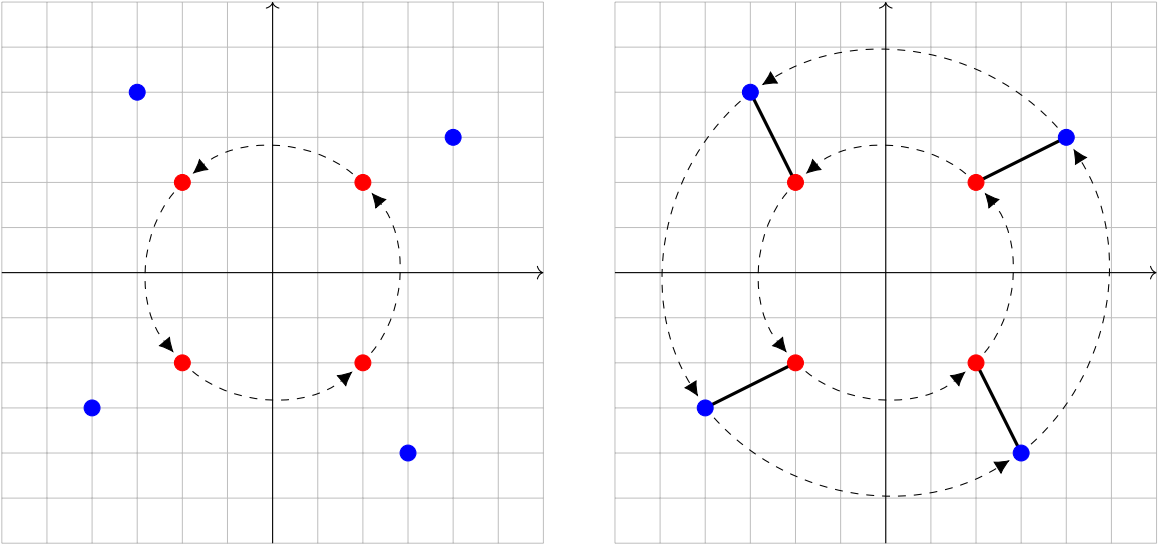
\begin{figure}
    \centering
    
    \begin{tikzpicture}[scale=0.6]
        \draw[very thin, lightgray] (-6, -6) grid (6, 6);
        
        \draw[->] (-6, 0) -- (6, 0) node[right] {};
        \draw[->] (0, -6) -- (0, 6) node[above] {};
        
        \coordinate (P1) at (4, 3);
        \coordinate (P2) at (-3, 4);
        \coordinate (P3) at (-4, -3);
        \coordinate (P4) at (3, -4);
    
        \filldraw[blue] (P1) circle (5pt) node[above] {};
        \filldraw[blue] (P2) circle (5pt) node[left] {};
        \filldraw[blue] (P3) circle (5pt) node[below] {};
        \filldraw[blue] (P4) circle (5pt) node[right] {};
        
        \coordinate (Q1) at (2, 2);
        \coordinate (Q2) at (-2, 2);
        \coordinate (Q3) at (-2, -2);
        \coordinate (Q4) at (2, -2);
        
        \filldraw[red] (Q1) circle (5pt) node[above] {};
        \filldraw[red] (Q2) circle (5pt) node[left] {};
        \filldraw[red] (Q3) circle (5pt) node[below] {};
        \filldraw[red] (Q4) circle (5pt) node[right] {};

    \draw[dashed, -{Latex[length=2mm,width=2mm]}, bend right=40, shorten >=5pt, shorten <=5pt] (Q1) to node[midway, above, sloped] {} (Q2);
    \draw[dashed, -{Latex[length=2mm,width=2mm]}, bend right=40, shorten >=5pt, shorten <=5pt] (Q2) to node[midway, above, sloped] {} (Q3);
    \draw[dashed, -{Latex[length=2mm,width=2mm]}, bend right=40, shorten >=5pt, shorten <=5pt] (Q3) to node[midway, above, sloped] {} (Q4);
    \draw[dashed, -{Latex[length=2mm,width=2mm]}, bend right=40, shorten >=5pt, shorten <=5pt] (Q4) to node[midway, above, sloped] {} (Q1);
    \end{tikzpicture}%
    \qquad%
    \begin{tikzpicture}[scale=0.6]

    \draw[very thin, lightgray] (-6, -6) grid (6, 6);
    
    \draw[->] (-6, 0) -- (6, 0) node[right] {};
    \draw[->] (0, -6) -- (0, 6) node[above] {};
    
    \draw[very thick, black,] (P1) -- (Q1);
    \draw[very thick, black,] (P2) -- (Q2);
    \draw[very thick, black,] (P3) -- (Q3);
    \draw[very thick, black,] (P4) -- (Q4);
    
    \coordinate (P1) at (4, 3);
    \coordinate (P2) at (-3, 4);
    \coordinate (P3) at (-4, -3);
    \coordinate (P4) at (3, -4);
    
    \filldraw[blue] (P1) circle (5pt) node[above] {};
    \filldraw[blue] (P2) circle (5pt) node[left] {};
    \filldraw[blue] (P3) circle (5pt) node[below] {};
    \filldraw[blue] (P4) circle (5pt) node[right] {};
    
    \coordinate (Q1) at (2, 2);
    \coordinate (Q2) at (-2, 2);
    \coordinate (Q3) at (-2, -2);
    \coordinate (Q4) at (2, -2);
    
    \filldraw[red] (Q1) circle (5pt) node[above] {};
    \filldraw[red] (Q2) circle (5pt) node[left] {};
    \filldraw[red] (Q3) circle (5pt) node[below] {};
    \filldraw[red] (Q4) circle (5pt) node[right] {};

    \draw[dashed, -{Latex[length=2mm,width=2mm]}, bend right=40, shorten >=5pt, shorten <=5pt] (Q1) to node[midway, above, sloped] {} (Q2);
    \draw[dashed,-{Latex[length=2mm,width=2mm]}, bend right=40, shorten >=5pt, shorten <=5pt] (Q2) to node[midway, above, sloped] {} (Q3);
    \draw[dashed,-{Latex[length=2mm,width=2mm]}, bend right=40, shorten >=5pt, shorten <=5pt] (Q3) to node[midway, above, sloped] {} (Q4);
    \draw[dashed,-{Latex[length=2mm,width=2mm]}, bend right=40, shorten >=5pt, shorten <=5pt] (Q4) to node[midway, above, sloped] {} (Q1);

    \draw[dashed,-{Latex[length=2mm,width=2mm]}, bend right=40, shorten >=5pt, shorten <=5pt] (P1) to node[midway, above, sloped] {} (P2);
    \draw[dashed,-{Latex[length=2mm,width=2mm]}, bend right=40, shorten >=5pt, shorten <=5pt] (P2) to node[midway, above, sloped] {} (P3);
    \draw[dashed,-{Latex[length=2mm,width=2mm]}, bend right=40, shorten >=5pt, shorten <=5pt] (P3) to node[midway, above, sloped] {} (P4);
    \draw[dashed,-{Latex[length=2mm,width=2mm]}, bend right=40, shorten >=5pt, shorten <=5pt] (P4) to node[midway, above, sloped] {} (P1);
    \end{tikzpicture}
    \caption{An illustration of the orbit pairing lemma (Lemma~\ref{lem:pairing}) and its proof. \textbf{(left)} A pair of orbits in $\bbR^2$ generated by a subgroup $C_4\leq \op{O}$(2). The permutation on the inner orbit induced by a generator of $C_4$ is depicted with dashed arrows. \textbf{(right)} Black line segments connect each point in the inner orbit with the nearest point in the outer orbit, representing the equivariant bijection $\beta$ constructed in the proof of Lemma~\ref{lem:pairing}. The equivariance of this bijection determines the action of the $C_4$ generator on the outer orbit as well, represented with dashed arrows.}
    \label{fig:pairing}
\end{figure}

The group $G$ acts on different orbits in concert, but the one-orbit theorem (Theorem~\ref{thm:one-orbit}) only enables us to recover $G$'s action on individual generic orbits.
The following lemma will allow us to compare and connect the actions of $G$ across multiple generic orbits.

\begin{lemma}[Orbit pairing lemma]
\label{lem:pairing}
	Let $V$ denote a finite-dimensional (real or complex) Hilbert space, and take any finite $G\leq \op{Aut}(V)$. Then for generic $(v,w)\in V^2$, any $\sigma \in \op{Aut}(Gv \cup Gw)$ sends $Gv$ to itself, and $\sigma$ is uniquely determined by its action on $Gv$.
\end{lemma}
\begin{proof}
    Define $\beta\colon Gv \to Gw$ by
	\[\beta(x) := \arg\min_{y\in Gw}\|x - y\|.\]
	We claim that for generic $(v,w)\in V^2$, $\beta$ is a well-defined bijection satisfying $\beta\circ \sigma = \sigma\circ \beta$ for any $\sigma \in \operatorname{Aut}(Gv \cup Gw)$.
    
    We first show $\beta$ is well-defined. It suffices to show that for a generic $(v,w)\in V^2$, there is a unique $g\in G$ that minimizes $\|v-gw\|^2$.
    Indeed, since $G$ is a group of linear isometries, for each $h,k\in G$ with $h\neq k$, we have $\|v-hw\|^2=\|v-kw\|^2$ if and only if $\langle v,hw\rangle=\langle v,kw\rangle$, if and only if $\langle v,(h-k)w\rangle=0$.
    Meanwhile, since $h\neq k$, $A:=h-k$ is nonzero, and so the polynomial $(v,w)\mapsto\langle v,Aw\rangle$ is nonzero, for example, at the top left- and right-singular vectors of $A$.
    Thus, the set $K_{h,k}$ of $(v,w)\in V^2$ for which $\|v-hw\|^2\neq\|v-kw\|^2$ is generic, and so $g\mapsto \|v-gw\|^2$ has a unique minimizer for every $(v,w)$ in the generic set $\bigcap_{h,k\in G,h\neq k}K_{h,k}$.
        
    Now let $\sigma \in \operatorname{Aut}(Gv\cup Gw)$. Then $\sigma$ extends to some linear isometry $M\colon V\to V$, letting us write
    \begin{align*}
        \beta\sigma(a) & = \arg\min_{b \in Gw} \|\sigma(a)-b\| = \arg\min_{b \in Gw} \|Ma-b\| = \arg\min_{b \in Gw} \|a-M^{-1}b\| = \arg\min_{b \in Gw} \|a-\sigma^{-1}(b)\| = \sigma\beta(a). 
    \end{align*}
    Note that, in particular, $\beta$ can be seen to be $G$-equivariant by considering the symmetries $\sigma$ that arise from the action of $G$. Finally, since we may assume by Lemma~\ref{lem.distinct points} that both orbits have size $|G|$, and the actions of $G$ on each orbit is transitive, it follows from $G$-equivariance that $\beta$ is bijective.

    With the properties of $\beta$ established, we will use it to finish the proof. We may assume that $\|v\| \neq \|w\|$ since $\|v\|^2 - \|w\|^2$ is a nonzero polynomial in $v$ and $w$, so the complement of its zero set is generic. Hence, the points in $Gv$ and $Gw$ have different norms, so $\sigma$ can be restricted to a map $\sigma|_{Gv}\colon Gv\to Gv$. The action of $\sigma$ on any $x\in Gw$ is then given by $\sigma(x) = \beta\sigma|_{Gv}\beta^{-1}(w)$,
    finishing the proof.
\end{proof}

In the following theorem, $[\bbC:F]$ denotes the degree of the field extension $\bbC/F$.
Before proceeding, we take a moment to demonstrate how this cheeky shorthand for ``1 if $F=\bbC$ and 2 if $F=\bbR$'' is not devoid of mathematical content.
In the theorem that follows, the ``1 if $F=\bbC$'' comes from the one-orbit theorem, and the ``2 if $F=\bbR$'' comes from the two-orbit theorem. Furthermore, the ``two'' of the two-orbit theorem arises in its proof when we view $\bbC$ as a real vector space of dimension two and reduce to the one-orbit theorem.
As such, our use of $[\bbC:F]$ does ultimately trace back to the degree of the field extension $\bbC/F$. We will use the notation $[\bbC:F]$ freely through the remainder of the paper.

\begin{theorem}[Multi-orbit theorem]
\label{thm:permutation-extension}
    Let $V$ denote a finite-dimensional Hilbert space over $F\in \{\mathbb R, \mathbb C\}$, and take any finite $G\leq \op{Aut}(V)$.
	For generic $(v_1,\dots,v_k) \in V^k$, the canonical map $G\to \op{Aut}(Gv_1\cup \dots \cup Gv_k)$ is an isomorphism if $k\geq [\mathbb C : F]$.
\end{theorem}
    
\begin{proof}
We argue in cases:

\medskip

\noindent
\textbf{Case I:} $k=1$ and $F=\mathbb C$. 
This case is covered by the one-orbit theorem (Theorem~\ref{thm:one-orbit}).

\medskip

\noindent
\textbf{Case II:} $k=2$. 
The $F=\mathbb R$ case follows from the two-orbit theorem (Theorem~\ref{thm:two-orbit}), and so we may assume $F=\mathbb C$.
Select $(v,w) \in V^2$ such that $|Gv| = |G|$ as in Lemma~\ref{lem.distinct points}, the one-orbit theorem (Theorem~\ref{thm:one-orbit}) holds for $v$, and the orbit pairing lemma (Lemma~\ref{lem:pairing}) holds for the pair $(v,w)$, all of which are generic conditions.
The canonical map $G\to \operatorname{Aut}(Gv \cup Gw)$ is injective since the action of $G$ on $Gv$ has trivial stabilizers by assumption.
Now consider any $\sigma \in \operatorname{Aut}(Gv\cup Gw)$. 
By Lemma~\ref{lem:pairing}, $\sigma$ acts on $Gv$ and is uniquely determined by that action. 
Since $\sigma|_{Gv} \in \op{Aut}(Gv)$, then by Theorem~\ref{thm:one-orbit}, we must have $\sigma|_{Gv} = g|_{Gv}$ for some $g\in G$. 
Hence, both $\sigma$ and $g|_{Gv\cup Gw}$ are automorphisms extending $\sigma|_{Gv}$.
By uniqueness, it follows that $\sigma = g|_{Gv\cup Gw}$. 

\medskip

\noindent
\textbf{Case III:} $k \geq 3$.
Suppose that $(v_1,\dots,v_k) \in V^k$ are such that, for all $2\leq j \leq k$, the $k=2$ case holds for the pair $(v_1,v_j)$. 
We also invoke Lemma~\ref{lem.distinct points} to assume $|Gv_j| = |G|$ for all $1\leq j \leq k$.
As in the $k=2$ case, the canonical map $G\to \operatorname{Aut}(Gv_1\cup \dots \cup Gv_k)$ is injective.
Consider an arbitrary $\sigma\in \operatorname{Aut}(Gv_1\cup \dots \cup Gv_k)$.
Invoking the $k=2$ case, $\sigma|_{Gv_1\cup Gv_j}$ is in the image of the canonical map $G\to \op{Aut}(Gv_1 \cup Gv_j)$ for each $2\leq j \leq k$. 
Hence, the restriction of $\sigma$ to $Gv_1\cup Gv_j$ satisfies $\sigma|_{Gv_1\cup Gv_j} = g_j|_{Gv_1\cup Gv_j}$ for some $g_j \in G$.
It suffices to show that $g_j = g_2$ for each $3\leq j \leq k$. 
Indeed, since $v_1$ is in both $Gv_1\cup Gv_2$ and $Gv_1\cup Gv_j$, it follows that $g_2v_1 = \sigma(v_1) = g_jv_1$. 
Since the action of $G$ on $Gv_1$ has trivial stabilizers by assumption, we conclude $g_j = g_2$.
\end{proof}

\subsection{Linearly extending the action}

By the multi-orbit theorem (Theorem~\ref{thm:permutation-extension}), given $k\geq[\mathbb{C}:F]$ generic orbits, we can recover all permutations of the points in those orbits that are realized by members of $G$.
By extending linearly, we can then reconstruct how $G$ acts on the invariant subspace $S$ spanned by these orbits.
Certainly if $S=V$, then at this point we have successfully constructed the transformations in $G$, but even if $S\subsetneq V$, we may still be able to recover $G$.
Indeed, if the codimension of $S$ is smaller than the dimension $r$ of the lowest-dimensional nontrivial representation of the abstract group $G$, then $G$ necessarily acts trivially on the orthogonal complement of $S$.
It remains to determine how many generic orbits are needed to span a subspace of codimension smaller than $r$.

Our bounds will involve a particular representation of the hidden group. Recall that for any finite group $G$, the \textbf{regular representation} of $G$ over a field $F$ is the vector space $R$ consisting of all formal $F$-linear combinations of elements of $G$.
Additionally, $R$ comes equipped with a $G$-action:
\[
g\cdot \sum_{h\in G}a_h h
= \sum_{h\in G}a_h gh
= \sum_{k\in G}a_{g^{-1}k} k.
\]
To see how the regular representation is relevant to our problem, note that an orbit $Gv$ can be viewed as the image of the $G$-equivariant map $g\mapsto gv$, which we can extend linearly to form a map $R\to V$ with image equal to $\operatorname{span}Gv$.
By passing to an equivariant linear map between representations, we gain access to tools from representation theory like Schur's lemma.
The following lemma extends this idea to unions of orbits, which are particularly relevant to our application:

\begin{lemma}\label{lem:regular-rep-maps}
Let $V$ denote a finite-dimensional vector space over a field $F$, and take any finite group $G$ with a linear action on $V$.
Let $R$ denote the regular representation of $G$ over $F$, and let $e_i$ denote the vector in $R^k$ which is the identity in the $i$th component and zero in all other components. 
\begin{enumerate}[label=(\alph*)]
\item
\label{item.reg rep b}
For every $v_1,\dots, v_k \in V$, there exists a $G$-equivariant linear map $\phi\colon R^k \to V$ such that 
\[
\phi(e_i) = v_i
\qquad
\text{for all}
\qquad
1\leq i \leq k.
\]
\item
\label{item.reg rep a}
For every $G$-equivariant linear map $\phi\colon R^k \to V$, it holds that
\[
\operatorname{im}\phi 
= \operatorname{span}\bigcup_{i=1}^k G\phi(e_i).
\]
\end{enumerate}
\end{lemma}

\begin{proof}
For \ref{item.reg rep b}, select any $v_1, \dots, v_k \in V$. 
For each $1\leq i \leq k$, define $\phi_i\colon R\to V$ by taking $\phi_i(g) = gv_i$ for $g\in G$ and extending linearly. 
Then $\phi\colon R^k \to V$ defined by
\[
\phi(r_1,\dots, r_k) 
= \phi_1(r_1) + \dots + \phi_k(r_k)
\]
satisfies the desired properties.
For \ref{item.reg rep a}, we note that $\bigcup_{i=1}^k Ge_i$ forms a basis for $R^k$, and so
\[
\operatorname{im}\phi 
= \operatorname{span}\phi\bigg(\bigcup_{i=1}^k Ge_i\bigg)
= \operatorname{span}\bigcup_{i=1}^k \phi(Ge_i)
= \operatorname{span}\bigcup_{i=1}^k G\phi(e_i),
\]
where the last step applies the fact that $\phi$ is $G$-equivariant. 
\end{proof}

Next, we apply Lemma~\ref{lem:regular-rep-maps} to characterize when $k$ generic $G$-orbits span a large subspace of $V$:

\begin{theorem}\label{thm:complex-orbit-span-equivalence}
Let $V$ denote a finite-dimensional vector space over a field $F\in\{\mathbb{R},\mathbb{C}\}$, and take any finite group $G$ with a linear action on $V$.
Let $r$ denote the dimension of the smallest nontrivial representation of $G$, and let $R$ denote the regular representation of $G$ over $F$. 
Then the following are equivalent:
\begin{enumerate}[label=(\alph*)]
\item\label{item:generic-span} 
For generic $(v_1,\dots,v_k)\in V^k$, it holds that $\operatorname{codim}\operatorname{span}\bigcup_{i=1}^k Gv_i<r$.
\item\label{item:exist-span}
There exists $(v_1,\dots, v_k) \in V^k$ such that $\operatorname{codim}\operatorname{span}\bigcup_{i=1}^k Gv_i<r$.
\item\label{item:exist-surjection} 
There exists a $G$-equivariant linear map $\phi\colon R^k \to V$ such that $\operatorname{corank}\phi<r$.
\item\label{item:k-large} 
It holds that 
\begin{equation}
\label{eq.threshold for repn}
k 
\geq \max_\pi \frac{n_\pi(V)-(r-1)\cdot[\pi=\mathbf{1}]}{n_\pi(R)},
\end{equation}
where the maximum is over all (finitely many) irreducible representations $\pi$ of $G$, $n_\pi(V)$ denotes the multiplicity of $\pi$ in $V$, $\mathbf{1}$ denotes the trivial representation, and the Iverson bracket $[P]$ is $1$ when the statement $P$ holds and $0$ otherwise.
\end{enumerate}
\end{theorem}

Recall that $n_\pi(R)=\dim\pi$ in the complex case, but the expression is more complicated in the real case.

\begin{proof}[Proof of Theorem~\ref{thm:complex-orbit-span-equivalence}]
\ref{item:generic-span} $\Leftrightarrow$ \ref{item:exist-span}:
The forward direction is immediate.
For the converse, put $d:=\operatorname{dim}V$ and consider the function $p\colon V^k \to \mathbb{R}$ defined by
\[
p(v_1,\dots, v_k) 
= \sum_{\substack{I\subseteq\{1,\ldots,d\}\\J\subseteq\{1,\ldots,k|G|\}\\|I|=|J|=d-r+1}}
|\operatorname{det}M(v_1,\ldots,v_k)_{I,J}|^2,
\]
where the columns of $M(v_1,\ldots,v_k)\in F^{d\times k|G|}$ are the coefficients (with respect to some fixed basis of $V$) of the vectors $gv_i$ for $g\in G$ and $1\leq i\leq k$ (in some fixed order).
Notably, the span of the vectors
\[
\set{~gv_i~}{~g\in G,~1\leq i\leq k~} 
= \bigcup_{i=1}^k Gv_i
\]
has codimension $<r$ precisely when $p(v_1,\dots, v_k) \neq 0$.
If we view $V$ as a real vector space, then $p$ is a polynomial function of $k\operatorname{dim}_{\mathbb{R}}(V)$ real coordinates.
As such, \ref{item:exist-span} implies that $p$ is a nonzero polynomial function, and so the complement of its zero set is the desired generic subset of $V^k$.

\ref{item:exist-span} $\Leftrightarrow$ \ref{item:exist-surjection}:
This follows from Lemma~\ref{lem:regular-rep-maps}. 
For~($\Rightarrow$), select $v_1,\ldots,v_k\in V$ such that the span of $\bigcup_{i=1}^k Gv_i$ has codimension $m$.
Then Lemma~\ref{lem:regular-rep-maps}\ref{item.reg rep b} delivers a $G$-equivariant linear map $\phi\colon R^k\to V$ of corank $m$ by Lemma~\ref{lem:regular-rep-maps}\ref{item.reg rep a}.
For~($\Leftarrow$), take any $G$-equivariant linear $\phi\colon R^k\to V$ corank $m$ and put $v_i:=\phi(e_i)$ for $1\leq i\leq k$.
Then the span of $\bigcup_{i=1}^k Gv_i$ has codimension $m$ by Lemma~\ref{lem:regular-rep-maps}\ref{item.reg rep a}.

\ref{item:exist-surjection} $\Leftrightarrow$ \ref{item:k-large}:
By Schur's lemma, there exists a $G$-equivariant linear map $R^k\to V$ of corank $<r$ if and only if the following hold simultaneously:
\begin{itemize}
\item[(i)]
$n_\pi(R^k) \geq n_\pi(V)$ for every nontrivial irreducible representation $\pi$ of $G$, and
\item[(ii)]
$n_\mathbf{1}(R^k) > n_\mathbf{1}(V)-r$.
\end{itemize}
Since $n_\pi(R^k) = kn_\pi(R)$, the result follows.
\end{proof}

We combine the multi-orbit theorem (Theorem~\ref{thm:permutation-extension}) with Theorem~\ref{thm:complex-orbit-span-equivalence} to identify conditions under which we can recover the concrete group $G$ from generic orbits.

\begin{corollary}
\label{cor.near-opt generic recovery}
Let $V$ denote a finite-dimensional Hilbert space over a field $F\in\{\mathbb{R},\mathbb{C}\}$, and take any finite $G\leq\operatorname{Aut}(V)$.
\begin{enumerate}[label=(\alph*)]
\item\label{item.determine repn a}
If $k$ does not satisfy \eqref{eq.threshold for repn}, then every combination of $k$ orbits of $G$ can be realized as orbits of another subgroup of $\operatorname{Aut}(V)$.
That is, $k$ orbits fail to determine the concrete group $G$.
\item\label{item.determine repn b}
If $k$ satisfies both $k\geq[\mathbb{C}:F]$ and \eqref{eq.threshold for repn}, then $k$ generic orbits determine the concrete group $G$.
\end{enumerate}
\end{corollary}

\begin{proof}
For \ref{item.determine repn a}, suppose $k$ does not satisfy \eqref{eq.threshold for repn}.
Then by Theorem~\ref{thm:complex-orbit-span-equivalence}, every choice of $k$ orbits spans an invariant subspace $S$ of codimension at least $r$.
Considering the abstract group $G$ has a nontrivial representation of dimension $r$, we cannot determine whether $G$ acts trivially on the orthogonal complement $S^\perp$ or nontrivially on an invariant subspace of $S^\perp$.
Overall, $k$ orbits fail to determine the concrete group $G$.

For \ref{item.determine repn b}, given $k\geq[\mathbb{C}:F]$ generic orbits, the multi-orbit theorem (Theorem~\ref{thm:permutation-extension}) determines all permutations of the points in those orbits that are realized by members of $G$.
By extending linearly, this determines how $G$ acts on the invariant subspace $S$ spanned by these orbits.
Since $k$ satisfies \eqref{eq.threshold for repn} by assumption, Theorem~\ref{thm:complex-orbit-span-equivalence} implies that the codimension of $S$ is smaller than the dimension $r$ of the smallest nontrivial representation of the abstract group $G$, and so $G$ necessarily acts trivially on the orthogonal complement $S^\perp$.
Overall, $k$ generic orbits determine the concrete group $G$.
\end{proof}

\begin{example}
For $C_n$ acting on $\mathbb{C}^n$ by cyclic permutations of coordinates, Corollary~\ref{cor.near-opt generic recovery}\ref{item.determine repn b} gives that a single generic orbit determines the representation.
More generally, a single generic orbit suffices whenever $V$ is the regular representation of $G$ over $\mathbb{C}$.
\end{example}

\begin{example}
A gap between \ref{item.determine repn a} and \ref{item.determine repn b} in Corollary~\ref{cor.near-opt generic recovery} appears when the bound in \eqref{eq.threshold for repn} equals $1$ but the field $F$ is real, i.e., $[\mathbb{C}:F]=2$.
This occurs, for example, when $V=\mathbb{R}^2$ and $G\neq C_2$, since then $n_\pi(V)\leq 1$ for all $\pi$.
This gap is an artifact of our analysis.
Indeed, given a generic orbit $S$ of a finite group $G\leq\operatorname{O}(2)$, Example~\ref{ex.real gram graph to G in R^2} explains how to recover $G$ up to isomorphism.
Furthermore, in these particular cases, one can show that $G$ is isomorphic to a unique subgroup of $\operatorname{Aut}(S)$, and this permutation action on $S$ extends linearly to $\operatorname{span}S=\mathbb{R}^2$.
Overall, $G$ is concretely determined by a single generic orbit.
We leave it as an open problem whether the hypothesis $k\geq[\mathbb{C}:F]$ in Corollary~\ref{cor.near-opt generic recovery}\ref{item.determine repn b} can be removed in general.
\end{example}

We conclude by identifying a shortcoming of Corollary~\ref{cor.near-opt generic recovery}, even in the complex case.

\begin{example}
Consider the group $G=\langle -\operatorname{id}\rangle\leq \operatorname{U}(d)$.
Corollary~\ref{cor.near-opt generic recovery} reports that $k$ orbits determine $G$ as a concrete group only if $k\geq d$, and furthermore, $d$ generic orbits suffice.
Paradoxically, we claim that a single generic orbit determines the concrete group in this case.
Indeed, a generic orbit under $G$ takes the form $\{\pm x\}$ for some nonzero $x\in\mathbb{C}^d$, while for any other order-$2$ subgroup of $\operatorname{U}(d)$, the centroid of a generic orbit is nonzero.
This does not contradict Corollary~\ref{cor.near-opt generic recovery}\ref{item.determine repn a} because the orbits considered there are not necessarily generic, whereas genericity allows one to make additional inferences.
\end{example}

\section{Discussion}
\label{sec.discussion}

In this paper, we presented bounds on the number of generic orbits required to abstractly or concretely recover a finite group of automorphisms of a finite-dimensional Hilbert space.
Opportunities for follow-on work remain.
First, in the real case, does a single generic orbit betray the group up to isomorphism? 
Is it possible to determine the group's permutation action on this generic orbit?
What is the minimum number of generic orbits necessary to recover a concrete group in general?
Next, what can be said if the group does not act by linear isometries?
It is fruitful to consider finite groups that preserve other bilinear forms?
While our methods for recovering the abstract group require an isometric linear action, we note that Theorem~\ref{thm:complex-orbit-span-equivalence} only requires a linear action.
Finally, how might one estimate groups in a way that is stable to noise, and how does noise impact the number of orbits needed to estimate the group?
These questions are particularly relevant to the task of learning symmetries from data.

\section*{Acknowledgments}

DGM was supported by NSF DMS 2220304.
BV thanks William C.\ Newman and various participants at the Norbert Wiener Center's 2024 Fall Fourier Talks for invaluable conversations about the ideas in this paper.
Both authors thank Adam Savage for highlighting the utility of deadlines.

\end{document}